\documentclass[11pt]{amsart}
\usepackage{geometry,verbatim,xcolor, pdfpages,caption}
\usepackage{graphicx,dsfont}
\usepackage{amssymb, mathtools}
\usepackage{bm}  

\usepackage{bbding}
\usepackage{pifont}
\usepackage{wasysym}

\DeclareMathOperator{\csch}{csch}

\title[Spine of Fleming-Viot process]{The spine of two-particle Fleming-Viot process\\ in a bounded interval}

\author{Krzysztof Burdzy}
\address{Department of Mathematics,
University of Washington,
Seattle, WA 98195, USA.}
\email{burdzy@uw.edu}
\thanks{K.B.'s research was supported in part by Simons Foundation Grants 506732 and 928958.}
\author{J\'anos Engl\"ander}
\address{Department of Mathematics, University of Colorado Boulder, Boulder, CO 80309, USA.}
\email{janos.englander@colorado.edu}
\thanks{J.E.'s research was supported in part by Simons Foundation Grant  579110.}
\author{Donald E. Marshall}
\address{Department of Mathematics, 
University of Washington,
Seattle, WA 98195, USA.}
\email{dmarshal@uw.edu}

\begin{document}

%\section{}
%\subsection{}

%%%%%%%%%%%%%%%%%%%%%%%%%%%%%%%%%%%%%%%%%%%%%%
%%                                          %%
%% Uncomment next line to change            %%
%% the type of equation numbering           %%
%%                                          %%
%%%%%%%%%%%%%%%%%%%%%%%%%%%%%%%%%%%%%%%%%%%%%%
\numberwithin{equation}{section}
%%%%%%%%%%%%%%%%%%%%%%%%%%%%%%%%%%%%%%%%%%%%%%
%%                                          %%
%% For Axiom, Claim, Corollary, Hypothesis, %%
%% Lemma, Theorem, Proposition              %%
%% use \theoremstyle{plain}                 %%
%%                                          %%
%%%%%%%%%%%%%%%%%%%%%%%%%%%%%%%%%%%%%%%%%%%%%%
\theoremstyle{plain}

\newtheorem{theorem}{Theorem}[section]
\newtheorem{corollary}[theorem]{Corollary}
\newtheorem{lemma}[theorem]{Lemma}
\newtheorem{proposition}[theorem]{Proposition}
\newtheorem{claim}[theorem]{Claim}

%%%%%%%%%%%%%%%%%%%%%%%%%%%%%%%%%%%%%%%%%%%%%%
%%                                          %%
%% For Assumption, Definition, Example,     %%
%% Notation, Property, Remark, Fact         %%
%% use \theoremstyle{remark}                %%
%%                                          %%
%%%%%%%%%%%%%%%%%%%%%%%%%%%%%%%%%%%%%%%%%%%%%%
\theoremstyle{remark}

\newtheorem{conjecture}[theorem]{Conjecture}
\newtheorem{definition}[theorem]{Definition}
\newtheorem{problem}[theorem]{Problem}
\newtheorem{remark}[theorem]{Remark}
\newtheorem{assumption}[theorem]{Assumption}
\newtheorem{step}{Step}
\setcounter{step}{0}

%%%%%%%%%%%%%%%%%%%%%%%%%%%%%%%%%%%%%%%%%%%%%%
%% Please put your definitions here:        %%
%%%%%%%%%%%%%%%%%%%%%%%%%%%%%%%%%%%%%%%%%%%%%%

\newcommand{\eps}{\varepsilon}
\newcommand{\vphi}{\varphi}

\newcommand{\calL}{\mathcal{L}}
\newcommand{\calF}{\mathcal{F}}
\newcommand{\calO}{\mathcal{O}}
\newcommand{\calG}{\mathcal{G}}
\newcommand{\calA}{\mathcal{A}}
\newcommand{\calD}{\mathcal{D}}
\newcommand{\calP}{\mathcal{P}}
\newcommand{\calT}{\mathcal{T}}
\newcommand{\calM}{\mathcal{M}}
\newcommand{\calR}{\mathcal{R}}
\newcommand{\calJ}{\mathcal{J}}
\newcommand{\calN}{\mathcal{N}}
\newcommand{\calS}{\mathcal{S}}
\newcommand{\calW}{\mathcal{W}}
\newcommand{\calK}{\mathcal{K}}
\newcommand{\calI}{\mathcal{I}}
\newcommand{\calC}{\mathcal{C}}
\newcommand{\calB}{\mathcal{B}}
\newcommand{\calV}{\mathcal{V}}
\newcommand{\calQ}{\mathcal{Q}}
\newcommand{\calH}{\mathcal{H}}

\newcommand{\I}{\mathds{1}}
\newcommand{\E}{\operatorname{\mathbb{E}}} % 
\newcommand{\Eb}{\operatorname{\mathbf{E}}}%
%Expectation
\renewcommand{\P}{\operatorname{\mathbb{P}}} % 
\newcommand{\Pb}{\operatorname{\mathbf{P}}}%
%Probability
\newcommand{\R}{\mathbb{R}}
\newcommand{\mQ}{\mathbb{Q}}
\newcommand{\C}{\mathds{C}}
\newcommand{\real}{\mathds{R}}
\newcommand{\N}{{\mathbb{N}}}
\newcommand{\Z}{{\mathbb Z}}
\newcommand{\Rn}{{\R^n}}
\newcommand{\Rt}{{\R^2}}
\newcommand{\prt}{\partial}

\newcommand{\ol}{\overline}
\newcommand{\wh}{\widehat}
\newcommand{\wt}{\widetilde}

\newcommand{\cV}{\mathcal{V}}
\newcommand{\cY}{\mathcal{Y}}

\newcommand\myeq{\stackrel{\mathclap{\normalfont\mbox{\tiny{sym}}}}{=}}

\newcommand{\comj}[1]{\textcolor{purple}{\texttt{Janos:} #1}}
\newcommand{\comk}[1]{\textcolor{blue}{\texttt{Chris:} #1}}
\newcommand{\comd}[1]{\textcolor{orange}{\texttt{Don:} #1}}

\def\Re{{\rm Re\,}}
\def\Im{{\rm Im\,}}

\newcommand{\bone}{\mathbf{1}}
\def\n{{\bf n}}
\def\bt{{\bf t}}
\def\bv{{\bf v}}
\def\bz{{\bf z}}
\def\bx{{\bf x}}
\def\ball{{\calB}}

\newcommand{\set}[1]{\left\{#1\right\}}

\newcommand{\e}{\mathrm e}
\newcommand{\F}{\mathcal F}
\newcommand{\FF}{\mathbb F}
\newcommand{\QQ}{\mathbb Q}
\newcommand{\G}{\mathcal G}
\renewcommand{\H}{\mathbf H}
\newcommand{\cH}{\mathcal H}
\newcommand{\cT}{\mathcal T}
\newcommand{\X}{\mathbf X}
\newcommand{\cX}{\mathcal X}
\newcommand{\x}{\mathbf x}
\newcommand{\bfa}{\mathbf a}
\newcommand{\Y}{\mathbf Y}
\newcommand{\V}{\mathbf V}
\newcommand{\bQ}{\mathbf Q}
\newcommand{\y}{\mathbf y}
\newcommand{\z}{\mathbf z}
\newcommand{\bfj}{\mathbf j}

\newcommand{\A}{\mathbf A}
\newcommand{\Aa}{\mathbb A}
\newcommand{\B}{\mathbf B}
\newcommand{\D}{\mathbb D}
\newcommand{\dd}[1]{\dot{#1}}
\newcommand{\cB}{\mathcal B}
\newcommand{\cA}{\mathcal A}
\newcommand{\cL}{\mathcal L}
\newcommand{\cP}{\mathcal P}
\newcommand{\cO}{\mathcal O}
\newcommand{\cM}{\mathcal M}
\newcommand{\cZ}{\mathcal Z}

\newcommand{\bl}{\mathbf 1}
\newcommand{\bfI}{\mathbf{I}}

\allowdisplaybreaks



\begin{abstract}
We show that the spine of the Fleming-Viot process driven by Brownian motion and starting with two particles in a bounded interval has a different law from that of  Brownian motion conditioned to stay in the interval forever. Furthermore, we estimate the ``extra drift.'' 
\end{abstract}
\maketitle

\section{Introduction}\label{section:introduction}

Our objective is to show that the spine of the Fleming-Viot process driven by Brownian motion and starting with two particles in the interval $(0,\pi)$ has a different law from that of a Brownian motion conditioned to stay in the interval forever.

A Fleming-Viot process is a process with a branching structure (but not a branching process according to the terminology adopted in the literature on branching processes). Under very mild assumptions, it has a unique spine, i.e. a trajectory within the branching structure that does not end before the lifetime of the process. When the number of individuals in the population is very large, the distribution of the spine is expected to be very close to the distribution of the driving process conditioned on survival forever; this has been proved in some cases (see the next section for references). There are already examples showing that the distribution of the spine may be different from the distribution of the driving process conditioned on survival forever. We believe that the example analyzed in this paper is more ``natural'' than the previously published ones.

\subsection{Literature review}
The following literature review is partly borrowed from \cite{KBTT}.
Fleming-Viot-type processes were originally defined in \cite{BHM}. 
In this model, there is a population of fixed size. Every individual moves independently from all other individuals according to the same Markovian transition mechanism, in a domain with a boundary. When an individual hits the boundary, the individual is killed and an individual chosen randomly (uniformly)  from the survivors  
splits into two individuals and the process continues in this manner. The question of whether the process can be continued for all times was addressed in \cite{BHM,extinctionOfFlemingViot,BBF,GK}. 
All of these papers studied, among other processes,
Fleming-Viot processes driven by Brownian motion.

Every Fleming-Viot process has a unique spine, i.e., a trajectory inside the branching tree that never hits the boundary of the domain where the process is confined; this was proved under strong assumptions in \cite[Thm. 4]{GK} and later in the full generality in \cite{BB18}. 

It was proved in \cite{BB18} that if the state space is finite and the number of individuals in the population goes to infinity then the distributions of spine processes converge to the distribution of the driving  Markov process conditioned on  survival forever. The same result has also been proven for if the driving process is a diffusion reflected normally off the boundary
of a compact domain with soft killing (see \cite{ToughArxiv}), or
 Brownian motion on a Lipshitz domain with hard killing (see \cite{BE.preprint}).

In \cite{BB18},
an example was given of a Fleming-Viot process driven by a Markov process on a three-element state space such that one of the elements plays the role of the boundary, the population consists of two individuals, and the distribution of the spine is not equal to the distribution of the driving Markov process conditioned on survival forever. A Markov process with a three-element state space seems to be a rather artificial example in the context of Fleming-Viot models. More recently, it was shown in \cite{KBTT} that the spine of the Fleming-Viot process with two individuals driven by Brownian motions on $[0,\infty)$ has a spine with a distribution different from the distribution of Brownian motion conditioned to stay positive, i.e., the distribution of the 3-dimensional Bessel process.
The fact that Brownian motion conditioned to stay positive forever is transient makes that example somewhat special. 
We hope that the example analyzed in this paper---of a two-particle Fleming-Viot process driven by Brownian motion in $(0,\pi)$---can be considered ``completely natural.'' 

There is a very exciting new development regarding two-particle Fleming-Viot processes, a recent paper \cite{MK}.
This new paper provides a technique that we use to prove one of our main theorems. It also strengthens motivation for studying this seemingly very specialized stochastic model.
The two main results in \cite{MK} are the following. First, for a two-particle Fleming-Viot process driven by Brownian motion in a bounded Euclidean domain, the uniform probability measure is stationary for the location of the particles at  branching times. Second, the Green function in the domain, considered as a function of two variables, is the stationary density for the Fleming-Viot process. Actually, the results in \cite{MK} are much more general---they are concerned with two-particle Fleming-Viot processes driven by any symmetric Markov processes.

\subsection{Heuristics}
Our objective is to show that the spine of the Fleming-Viot process driven by Brownian motion on $(0,\pi)$ has a different law from that of Brownian motion conditioned to stay in the interval forever.

The basic idea is to represent the two  particles moving in $(0,\pi)$ before one of them exits the interval as  a single two-dimensional Brownian particle until 
it hits the boundary of $(0,\pi)^2$. Without loss of generality, we will assume that the vertical component exits $(0,\pi)$ first. Then, until the exit time, the horizontal component represents the spine. We will use the following two facts about conditioned Brownian motion. First, one-dimensional Brownian motion conditioned to never exit $(0,\pi)$  is a space-time Doob's $H$-process (for an appropriate parabolic function $H$). Second, two-dimensional
Brownian motion conditioned to exit  $(0,\pi)^2$ via the upper or lower side is a Doob's $h$-process for an appropriate harmonic function.

\subsubsection{Remarks on drift versus particle numbers}
Consider $n$-particle Fleming-Viot process driven by Brownian motion in $(0,\pi)$.
One might be interested in how the drift of the spine is compared to that of  Brownian motion conditioned to stay in the interval forever, for a general particle number $n\ge 2$. Let $\tau^{\mathsf{settle}}_n$ denote the smallest time $s$ such that given the Fleming-Viot process's realization on $[0,s]$, the spine trajectory on $[0,t]$ can be deduced without ambiguity for all sufficiently small $t's$. Note that  the spine is launched from the unique particle with an  infinite genealogical tree. Thus,  $\tau^{\mathsf{settle}}_n$ is the first time when all but one of these trees become extinct.

In order to develop an intuition about the effect of the particle number on the drift, first consider $n=2$, as in the setup of this paper.
In this case, $\tau^{\mathsf{settle}}_2$ agrees with the first hitting time of the boundary by any of the two particles. Indeed,  if we start with  ``blue'' and  ``red'' particles, then at the  moment the first hit happens by, say, the blue particle, we know that the path of the red  particle on any smaller time interval is surely that of the spine.
In terms of conditioning, if we want to describe the law of the spine's trajectory on a time interval $[0,dt]$ then the conditioning  is not to survive forever, but rather until the other particle hits the boundary.   This weaker conditioning results in an inward drift that is weaker than that for  Brownian motion conditioned to stay in the interval forever.

Next, let $n>2$. In this situation, in order to conclude that  the spine trajectory on $[0,dt]$ is that of the red particle, it is no longer enough for the red particle to survive until another particle hits the boundary; one has to wait until $\tau^{\mathsf{settle}}_n$ to make such an inference. It is known (see \cite{BB18,GK}) 
 that $\tau^{\mathsf{settle}}_n<\infty$  almost surely for each fixed $n\ge 2$, but $\lim_{n\to\infty}\tau^{\mathsf{settle}}_n=\infty$ in law.
 This means that for each given $n\ge 2$ one has weaker conditioning than perpetual survival (namely, survival until $\tau^{\mathsf{settle}}_n$, or equivalently, until the extinction of the genealogical trees of all other particles), however  in the $n\to\infty$ limit the discrepancy fades away, and the condition indeed becomes perpetual survival, in accordance with the main result in \cite{BE.preprint}. Accordingly, the same can be said about the corresponding (inward) drifts of the spines at time zero.

\subsection{On software assisted proofs}
As a matter of principle, we tried to provide proofs that are human-verifiable. However, we used {\sf Mathematica} to obtain certain explicit formulas, exact quantities and approximate quantities.  We indicated essential uses of {\sf Mathematica} in the text.

 \subsection{Organization of the paper}
The rest  of the  paper  starts with rigorous statements of the model and the three main theorems in Section \ref{model}.
The theorems are proved in the remaining sections. Section \ref{midpoint} is based on Fourier analysis. Next comes Section \ref{conformal} presenting an argument based on complex analysis. Finally, Section \ref{sec:stat} contains arguments using results from \cite{MK}, so ultimately it is based on potential theoretic techniques.

\section{Model and main results}\label{model}

We will now define a Fleming-Viot process and  other elements of the model.
Informally, the process consists of two independent Brownian particles starting at the same point in $(0,\pi)$. At the time when one of them hits 0 or $\pi$, it is killed and the other one branches into two particles. The new particles start moving as independent Brownian motions and the scheme is repeated.

\subsection{Notation and definitions}
Let
$(W_1(t):t\geq 0)$ and $(W_2(t):t\geq 0)$ be two independent Brownian motions with $W_1(0) =a$, $W_2(0)=b$, $a,b\in(0,\pi)$. Let
\begin{align*}
T_0&=0,\\
\tau^1_j &= \inf\{t\geq 0: W_j(t) =0\text{  or  }\pi\}, \qquad j=1,2,\\
T_1&=\min(\tau^1_1,\tau^1_2),\\
m_1&=j\in \{1 , 2\} \text{  such that  }  \tau^1_j \ne T_1,\\
Y_1&=W_{m_1}(T_1),
\end{align*}
and for $k\geq 2$,
\begin{align*}
\tau^k_j&=\inf\{t>T_{k-1} : W_j(t)-W_j(T_{k-1})+Y_{k-1}=0\text{  or  }\pi\}, \qquad j=1,2,\\
T_k&=\min(\tau^k_1,\tau^k_2),\\
m_k&=j \in \{1 , 2\} \text{  such that  }  \tau^k_j \ne T_k,\\
Y_k&=W_{m_k}(T_k)-W_{m_k}(T_{k-1})+Y_{k-1}.
\end{align*}
Note that for a fixed $j$, $\tau^k_j$ need not be different for different $k$.

It follows from \cite[Thm.~5.4]{BBF} or \cite[Thm.~1]{GK} that, a.s.,
$ T_k\to \infty$.
Hence, for any $t\geq 0$ we can find $j\geq 1$ such that $t\in [T_{j-1},T_{j})$. Then we set
\begin{align}\label{j18.1}
\cV(t)&=(V_1(t),V_2(t))
      =(W_1(t)-W_1(T_{j-1})+Y_{j-1},W_2(t)-W_2(T_{j-1})+Y_{j-1}).
\end{align}
This completes the definition of $\{\cV(t), t\geq 0\}$, an example of a Fleming-Viot process. 
We will write $\Pb_{\mu}$ and $\Eb_{\mu}$ to
denote the probability and expectation associated with $\mathcal{V}$, where
 $\mu=\delta_a+\delta_b$.

Let $J_t = J(t)$ denote the spine, i.e., $J_t = V_1(t)$ for
$t\in [T_{k-1},T_{k})$ 
if $V_1(T_k-) \in(0,\pi)$.
If the last condition fails, we let  $J_t = V_2(t)$ for
$t\in [T_{k-1},T_{k})$.

Let $\lambda>0$ and $\phi$ denote the principal Dirichlet eigenvalue and eigenfunction on $(0,\pi)$ for the operator $-\frac12 \Delta$, i.e., $\phi(x) =\sin x$ and $\lambda = 1/2$. 
The function $H(x,t) := e^{\lambda t}\phi(x)$ is parabolic. If $W_t$ is one-dimensional  Brownian motion and we condition the space-time process $(W_t,t)$ using Doob-conditioning with $H$ then we obtain a process $(X_t,t)$ whose first component is ``Brownian motion conditioned to stay $(0,\pi)$ forever.'' The diffusion coefficient of $X$ is the same as for Brownian motion and the drift at $x\in (0,\pi)$ is $\phi'(x)/\phi(x)=\cot(x)$. 

Let $h$ be the harmonic function in $(0,\pi)^2$ with 0 boundary values on the left and right sides and boundary values 1 on the other two sides  (see \eqref{eq: BVP_h}). Let $\cV^h(t) =(\cV^h_1(t), \cV^h_2(t))$ denote the process $\{\cV(t), 0< t < T_1\}$ conditioned by $m_1=1$, i.e., $\cV_2$ exits from $(0,\pi)$ before $\cV_1$ does. The process $\cV^h$ is Doob's $h$-process. We will write $h_x(x,y) = \frac {\prt} {\prt x}h(x,y)$. 
The diffusion coefficient of $\cV^h_1$ is that of Brownian motion and its drift at $(x,y)\in(0,\pi)^2$ is $h_x(x,y)/h(x,y)$. Note that the drift depends on $\cV^h_2$.

Let $\P^{BM}_{(a,b)}$ be the distribution of $\{(W_1(t),W_2(t)), 0\leq t < T_1\}$ starting from $(a,b)$ and let $\P^h_{(a,b)}$ be the distribution of $\{(W_1(t),W_2(t)), 0\leq t < T_1\}$ conditioned by $\{\tau^1_2<\tau^1_1\}$. The latter distribution is an example of Doob's $h$-transform.
Note that 
$h(x,y)=\P^{BM}_{(x,y)}\left(\tau^1_2<\tau^1_1\right).$
 
\subsection{Differentiating conditional semigroups}
\begin{lemma}\label{le: cheating.okay}
If $a,b\in(0,\pi)$, $\mu=\delta_a+\delta_b$ and $f\in C_c^2((0,\pi))$ then
\begin{align*}
\lim\limits_{t\to 0} \frac{1}{t}\left(\Eb_{\mu}[f(J_t)]-\Eb_{\mu}[f(J_t)\mid T_1>t]\right)=0.
\end{align*}
\end{lemma}
\begin{proof}
One has
\begin{align*} 
\frac{1}{t}&\left(\Eb_{\mu}[f(J_t)]-\Eb_{\mu}[f(J_t)\mid T_1>t]\right)\\
&=\frac{1}{t}\frac{1}{\Pb_{\mu}\left(T_1>t\right)}\left(\Eb_{\mu}[f(J_t)]\Pb_{\mu}\left(T_1>t\right)-\Eb_{\mu}[f(J_t); T_1>t]\right)\\
&=\frac{1}{t}\frac{1}{\Pb_{\mu}\left(T_1>t\right)}\left(\Eb_{\mu}[f(J_t)]-\Eb_{\mu}[f(J_t)]\Pb_{\mu}\left(T_1\le t\right)-\Eb_{\mu}[f(J_t); T_1>t]\right)\\
&=\frac{1}{t}\frac{1}{\Pb_{\mu}\left(T_1>t\right)}\left(-\Eb_{\mu}[f(J_t)]\Pb_{\mu}\left(T_1\le t\right)+\Eb_{\mu}[f(J_t); T_1\le t]\right).
\end{align*}
As $\lim_{t\to 0}\Pb_{\mu}\left(T_1>t\right)=1$, it is enough to check that
$$\lim_{t\downarrow 0}\frac{1}{t}\left(\Eb_{\mu}[f(J_t); T_1\le t]-\Eb_{\mu}[f(J_t)]\Pb_{\mu}\left(T_1\le t\right)\right)=0.$$
Note that $\lim_{t\to 0}\Pb_{\mu}\left(T_1\le t\right)/t=0$
because if $(a-\epsilon, a+\epsilon)\subset (0,\pi)$ the probability of  hitting $(a-\epsilon, a+\epsilon)$ by any of the two particles by $t$  is exponentially small in $1/t$.
Since $f$ is bounded,  we are done.
\end{proof}

We have 
\begin{align*}
\E^h_{(a,b)}[f(W_1(t))\mid T_1>t]&=
\Eb_{\mu}[f(J_t)\mid \tau^1_1>\tau^1_2>t]\\
&=\Eb_{\mu}[f(J_t)\mid T_1>t,\ J_s=W_1(s),\ 0\le s<T_1],\\
\E^{1-h}_{(a,b)}[f(W_2(t))\mid T_1>t]&=
\Eb_{\mu}[f(J_t)\mid \tau^1_2>\tau^1_1>t]\\
&=\Eb_{\mu}[f(J_t)\mid T_1>t,\ J_s=W_2(s),\ 0\le s< T_1].
\end{align*}
An argument similar to the proof of Lemma \ref{le: cheating.okay} shows that 
\begin{align*}
\lim_{t\to 0}\frac{1}{t}\left(\E^h_{(a,b)}[f(W_1(t))\mid T_1>t]-f(a)\right)=\lim_{t\to 0}\frac{1}{t}\left(\E^h_{(a,b)}[f(W_1(t))]-f(a)\right).
\end{align*}
The right-hand side exists and equals
$$\frac{1}{2}f''(a)+\frac{h_x}{h}(a,b)
f'(a).
$$

\subsection{Main results}
In order to demonstrate that $J$ is not $H$-transformed Brownian motion, we will show that for all $0<a,b< \pi/2$,
\begin{align*}
&\lim_{t\to 0}\frac{1}{t}\left(\Eb_{\mu}[f(J_t)]-f(a)\mid J_s=W_1(s),\ 0\le s<T_1\right)=
\lim_{t\to 0}\frac{1}{t}\left(\E^h_{(a,b)}f(W_1(t))-f(a)\right)\\
&\quad \neq \frac{1}{2}f''(a)
+\frac{\phi'}{\phi}(a)f'(a).
\end{align*}
More precisely, the following two theorems hold. Recall that $\phi(x)=\sin(x)$, so $
\frac{\phi'}{\phi}(x)=\cot(x)$.
\begin{theorem}\label{prop:new1}
\begin{align}\label{j19.3}
\frac{\phi'}{\phi}(\pi/4)- \frac{ h_x}{h}(\pi/4,\pi/4)
&> 0,\\
\frac{\phi'}{\phi}(\pi/4)- \frac{ h_x}{h}(\pi/4,\pi/4)
& \approx 0.248532.\label{j19.4}
\end{align}
\end{theorem}

\begin{theorem}\label{j19.7}
For $0<x,y<\pi/2$,
\begin{align*}%\label{j14.3}
\frac{\phi'}{\phi}(x) >\frac{h_x(x,y)}{h(x,y)} .
\end{align*}
By symmetry, analogous inequalities hold in the other parts of $[0,\pi]^2$.
\end{theorem}

The following result shows the difference between the two laws in a different way.
\begin{theorem}\label{s26.1}
Suppose that $J_0=X_0$, a.s.
The distributions of the spine $\{J_t, 0\leq t <\infty\}$ and Brownian motion conditioned to stay in $(0,\pi)$ forever, i.e., $\{X_t, 0\leq t <\infty\}$ are singular.
\end{theorem}

The proofs of the three theorems will be given in Sections  \ref{midpoint}, \ref{conformal} and \ref{sec:stat}. 

\section{Drift comparison at a branching time at ``midpoint''}\label{midpoint}

We will use Fourier techniques to compare the drifts of the spine at a branching time and Brownian motion conditioned to 
stay in $(0,\pi)$ forever. In this section, we will limit our calculations to a specific position of the particles at the branching time, namely, $\pi/4$. 

\begin{proof}[Proof of Theorem \ref{prop:new1}]

The  function $h$ is the unique bounded solution to the boundary value problem 
\begin{align}\label{eq: BVP_h}
\begin{cases}
  \Delta h=0    \text{ on } (0,\pi)\times (0,\pi),\\
 \lim_{x\to 0}h(x,y)=\lim_{x\to \pi}h(x,y)=0,\  y\in (0,\pi),\\  
\lim_{y\to 0}h(x,y)=\lim_{y\to \pi}h(x,y)=1,\  x\in (0,\pi).\end{cases}
\end{align}
Existence follows from Theorem 1.3 in \cite[page 5]{GM}. Uniqueness follows from Lindel\"of's maximum principle, Lemma 1.1 in \cite[page 2]{GM}. 

Using the method of separation of variables, the solution has the following infinite series representation (see e.g., Lec. 34 in \cite{AO}).
 
For $n\geq 1$ and $x,y\in(0,1)$, let
\begin{align*}
F_n(y)&=\cosh(ny)+\cosh(n(\pi-y)),\\
G_n(y)&=\sinh(ny)+\sinh(n(\pi-y)),\\
a_n&=(2/\pi)\int_0^{\pi}\sin (nx)\, dx=\frac{2}{\pi}\cdot\frac{1-(-1)^n}{n},\\
\cX_n(x)&=\sin(nx) ,\\
\cY_n(y)&=a_n (F_n(y)-\coth(n\pi)G_n(y)).
\end{align*}
Then
\begin{align*}
h(x,y)&=\sum_{n=1}^{\infty} \cX_n(x)\cY_n(y).
\end{align*}

Let $\gamma_k=\coth((2k+1)\pi)$ and note that $a_{2k}=0$, while $a_{2k+1}=\frac{1}{\pi}\cdot\frac{4}{2k+1}$. We have
\begin{align}\notag
h(x,y)&=\frac{4}{\pi}\sum_{k=0}^{\infty} \frac{\sin((2k+1)x)}{2k+1}\left[F_{2k+1}(y)-\gamma_kG_{2k+1}(y)\right],\\
 h_x(x,y)&=\frac{4}{\pi}\sum_{k=0}^{\infty} \cos((2k+1)x)\left[F_{2k+1}(y)-\gamma_kG_{2k+1}(y)\right],\notag\\
 h_x(\pi/4,\pi/4)&=\frac{4}{\pi}\sum_{k=0}^{\infty} \cos((2k+1)\pi/4)\left[F_{2k+1}(\pi/4)-\gamma_kG_{2k+1}(\pi/4)\right]. \label{j19.5}
\end{align}

Note that 
\begin{align}\label{j19.1}
\{\sqrt{2}\cos((2k+1)\pi/4), k=0,1,2,\dots\}
= 1,-1,-1,1,1,-1,-1,1,1,-1,-1,1,1, \dots
\end{align}

We extend the definitions of $F_n(y)$ and $G_n(y)$ to  positive real values of the parameter $n$ and we let $c_t=F_{t}(\pi/4)-\coth(t\pi) G_{t}(\pi/4)$ for $t>0$.

\begin{lemma}
The function $t\to c_t$ is decreasing.
\end{lemma}

\begin{proof}
Using  {\sf Mathematica},
\begin{align*}
\frac{dc_t}{dt}&=\frac{\pi }{4}\sinh \left(\frac{\pi t}{4}\right)+\pi \csch ^2\left(\pi t\right)\sinh \left(\frac{\pi t}{4}\right)-\coth \left(\pi t\right)\frac{\pi }{4}\cosh \left(\frac{\pi t}{4}\right)\\
&\qquad+\frac{3\pi }{4}\sinh \left(\frac{3\pi t}{4}\right)+\pi \csch ^2\left(\pi t\right)\sinh \left(\frac{3\pi t}{4}\right)-\coth \left(\pi t\right)\frac{3\pi }{4}\cosh \left(\frac{3\pi t}{4}\right)\\
&=-\frac{1}{4} \pi  \sinh \left(\frac{\pi  t}{4}\right) \left(\cosh \left(\frac{\pi  t}{2}\right)+2\right) \text{sech}^2\left(\frac{\pi  t}{2}\right).
\end{align*}
The last expression is negative for $t>0$.
\end{proof}

The lemma and \eqref{j19.1} imply that  for every $j\geq0$,
\begin{align}\label{a21.1}
\sum_{k=1+4j}^{4+4j} \cos((2k+1)\pi/4)\left[F_{2k+1}(\pi/4)-\gamma_kG_{2k+1}(\pi/4)\right]<0.
\end{align}
Grouping the terms in \eqref{j19.5} by four, starting with $k=1$, one obtains a Leibniz-type series.
This observation and \eqref{a21.1} applied with $j=0$ yield
\begin{align*}
 h_x(\pi/4,\pi/4)&<S:=
\frac{4}{\pi} \cos(\pi/4)\left[F_{1}(\pi/4)-\coth(\pi)
G_{1}(\pi/4)\right]\\
&= \frac{2 \sqrt{2}} \pi \cosh \left(\frac{\pi }{4}\right) \text{sech}\left(\frac{\pi }{2}\right)
\approx 0.475282.
\end{align*}
The results on the last line were obtained using Mathematica.
By symmetry, $h(\pi/4,\pi/4)=1/2$, so
\begin{align*}
\cot&(\pi/4)- \frac{ h_x}{h}(\pi/4,\pi/4)
= 1 - 2  h_x(\pi/4,\pi/4)
> 1-2 S\\
&=1-2 
 \frac{2 \sqrt{2}} \pi \cosh \left(\frac{\pi }{4}\right) \text{sech}\left(\frac{\pi }{2}\right)
\approx 1-2 \cdot 0.475282 =0.0494362
> 0.
\end{align*}
This completes the proof of \eqref{j19.3}.

The approximation stated in \eqref{j19.4} was computed using {\sf Mathematica}. A partial sum up to  $k=5$ in \eqref{j19.5} gives the accuracy of six significant digits.
\end{proof}

\section{Analysis of a harmonic function via conformal mappings}\label{conformal}

 The Dirichlet problem \eqref{eq: BVP_h} can be transplanted from the square to any convenient Jordan domain by means of a conformal map.
(A bounded region in $\C$ whose boundary is a closed Jordan curve is called a Jordan domain). Indeed, if $\Omega$ is a Jordan domain, then there is a conformal map $f$ of the unit disk $D$ onto $\Omega$ which extends to be a homeomorphism of their closures,  $\overline{D}$ onto 
$\overline{\Omega}$, by the Riemann mapping theorem and Carath\'eodory's theorem  \cite[page~13]{GM}.  For example, if $h$ is the solution to \eqref{eq: BVP_h} on the square $S$ then $h\circ f$ is the solution on the disk with boundary values $0$ and $1$ on the (open) arcs of $\partial D$ corresponding to the boundary edges of $S$. Here we use the fact that a harmonic function on $S$ is the real part of an analytic function, and hence $h\circ f$ is also harmonic.
   In this section, we will transfer our problem from the square $(0,\pi)^2$ to the disk. The southwest quarter of the square corresponds to the southwest quadrant of the disk under our map. We examine partial derivatives on the disk instead of the square by the chain rule, but simplify their expressions using a conformal map of the quadrant onto a chosen subset of the right half plane. It is this transformation that allows us to make the somewhat delicate estimates of the partial derivatives.

 We will use both complex and real notation---the meaning should be obvious from the context.

The following claim is a special case of the Schwarz-Christoffel formula (see \cite{Ahlfors}, Ch. 6, Sect. 2.2, Exercise 5). Let $D:= \{z\in \C: |z|<1\}$.
The analytic function
\begin{align}\label{j12.8}
\vphi(z)= \int_0^z \frac 1 {\sqrt{1+w^4}} dw
\end{align}
is a conformal map of $D$ onto a square with vertices $\vphi(e^{i\pi/4}), \vphi(e^{i3\pi/4}), \vphi(e^{i5\pi/4}), \vphi(e^{i7\pi/4})$.

We are using the function $\int_0^z \frac 1 {\sqrt{1+w^4}}\, dw$ instead of the more traditional $\int_0^z \frac 1 {\sqrt{1-w^4}}\, dw$ to avoid a rotation.

\begin{remark}\label{rem:1}

If $z=r e^{ik\pi/4}$ then $$\varphi(z)=\left(\frac{1}{r}\int_0^{r} \frac 1 {\sqrt{1\pm t^4}}\, dt\right)\, z,$$
where $+$ corresponds to even $k$ and $-$ to odd $k$. That is,  rays in $D$
emanating from $0$ at angles that are multiples of $\pi/4$ are mapped by $\varphi$
onto rays in the square with the same slopes.
\end{remark}

\begin{lemma}\label{j14.4}
Let $c>0$ be such that $c\vphi$ maps the unit disc onto $(-\pi/2,\pi/2)^2$.
Then
\begin{align}
&c= 4\pi^{3/2}(\Gamma(1/4))^{-2}\approx 1.69443,
\end{align}
where $\Gamma$ is Euler's  Gamma function. 
\end{lemma}

\begin{proof}
We apply Remark \ref{rem:1} with $r,k=1$ to see that the diagonal of 
$(-\pi/2,\pi/2)^2$ is equal to 
\begin{align}\label{j14.10}
 \sqrt{2} \pi 
 = 2c \int_0^{1 } \frac 1 {\sqrt{1- t^4}} dt.
\end{align}
Substituting $u=1-t^4$, it follows that
\begin{align}\label{j14.11}
2c\int_0^{1} \frac {1}{\sqrt{1- t^4}}dt
&=\frac{c}{2}\int_0^1 u^{-1/2}(1-u)^{-3/4} du
 =\frac{c}{2} B(1/2,1/4),
 \end{align}
where $B$ is the Beta function. By (\ref{j14.10}) and (\ref{j14.11}), using the elementary identities $B(z,w)=\Gamma(z)\Gamma(w)/\Gamma(z+w)$,
$\Gamma(w)\Gamma(1-w)=\pi/\sin \pi w$, with $z=\frac12$, $w=\frac14$, and $\Gamma(\frac12)=\sqrt{\pi}$
it follows that $c=4\pi^{3/2}/\Gamma(1/4)^2.$ See p. 46 in \cite{Encyc}. The numerical value of $\Gamma(1/4)$ is well-known and can be computed using Gauss' algorithm for arithmetic–geometric mean (AGM).
\end{proof}

Recall that $h$ is the bounded harmonic function in $(0,\pi)^2$ with boundary values $0$ on the left and right sides, and $1$ on the lower and upper sides.

For any function $f(z)$ of complex variable $z=x+iy$, we will denote partial derivatives by $f_x,f_y, f_{xy}$, etc.

\begin{proposition}\label{j30.1}
For a fixed $0<x<\pi/2$, the function
\begin{align*}%\label{j14.3}
y \to \frac{h_x(x+iy)}{h(x+iy)} 
\end{align*}
is increasing for $0<y<\pi/2$.
\end{proposition}

\begin{proof}

Let $u$ be bounded and harmonic in $D$ with boundary values
\begin{align*}
u(e^{i\theta})=
\begin{cases}
0& \text{for  } -\pi/4 < \theta< \pi/4\text{  or  } 3\pi/4 < \theta< 5\pi/4,\\
1& \text{for  } \pi/4 < \theta< 3\pi/4\text{  or  } 5\pi/4 < \theta< 7\pi/4.
\end{cases}
\end{align*}

We find a formula for $u$ as follows. The function $z\to z^2$ maps $D$ onto itself and maps the boundary values of $u$ to the following boundary values: $0$ for $z\in \prt D$ with $\Re z >0$, and   $1$ for $z\in \prt D$ with $\Re z <0$. Next, we use the mapping $z \to (i+z)/(i-z)$ to map $D$ onto the right halfplane. Indeed for $z=x+iy,\ x,y\in\R,\ x^2+y^2=1$, we have 
$(i+z)/(i-z)=-i x/(y-1)$. The right half of the circle $\prt D$ is mapped onto the lower half of the vertical axis and the left half of $\prt D$ is mapped onto the upper half of the vertical axis. Hence, the boundary values are mapped onto $0$ on the lower half of the vertical axis and $1$ on the upper part of the vertical axis. A bounded harmonic function in the right halfplane with these boundary values is $z\to (1/\pi) \arg z + 1/2  
=\mathrm{Re}(-(i/\pi) \log z+1/2)$.

It follows that, for $z\in D$,
\begin{align}\label{j17.1}
u(z) = \Re\left(-\frac i \pi \log\left(\frac {i+z^2}{i-z^2}\right) + 1/2\right).
\end{align}
For $z\in D$, let
\begin{align*}%\label{j12.4}
f(z) = -\frac i \pi \log\left(\frac {i+z^2}{i-z^2}\right) + 1/2. 
\end{align*}
Then
\begin{align}\label{j12.3}
f'(z) 
= -\frac i \pi \left( \frac {2z}{i+z^2} + \frac {2z}{i-z^2}\right)
= - \frac 4 \pi \frac z {1+z^4}.
\end{align}
By the Cauchy-Riemann equations
$f'(z) = f_x(z)=u_x(z) -i u_y(z)$. 

Next we map the square $[0,\pi]^2\subset \R^2=\C$ onto $D$ as follows. The mapping $z\to z- \pi(1+i)/2$ maps the square $[0,\pi]^2$ onto the square with vertices $\pi(1+i)/2, \pi(-1+i)/2, -\pi(1+i)/2, \pi(1-i)/2 $. Hence,
\begin{align}\label{j12.1}
\psi(z) = \vphi^{-1} \left( \frac {z- \pi(1+i)/2}{c}\right)
\end{align}
maps $[0,\pi]^2$ onto $D$, where $c>0$ is defined in Lemma \ref{j14.4}.

Note that $h = u \circ \psi$ because $ u \circ \psi$ is a bounded harmonic function in $[0,\pi]^2$  with boundary values $0$ on the left and right sides, and $1$ on the lower and upper sides.

If $w=\psi(z)$ then, using \eqref{j12.1}, 
\begin{align}\label{j12.7}
\psi^{-1}(w) &= c \vphi(w) + \pi(1+i)/2,\\
(\psi^{-1})'(w)&= c \vphi'(w)=\frac c {\sqrt{1+w^4}}.\label{j12.30}
\end{align}

By the Cauchy-Riemann equations, $(f\circ \psi)'(z) = ((f'\circ \psi)(z)) \psi'(z) = h_x(z) - i h_y(z) $, for  $z\in [0,\pi]^2$. Thus by \eqref{j12.3}, for $z\in D$,
\begin{align}\label{j12.31}
(h_x\circ\psi^{-1})(z) - i (h_y\circ\psi^{-1})(z)
&= \frac
{- \frac 4 \pi \frac z {1+z^4}}
{c /\sqrt{1+z^4}}
= -\frac {4}{c\pi}
\frac z
{\sqrt{1+z^4}}.\ 
\end{align}
We square both sides and compute the imaginary part,
\begin{align}
-2(h_x\circ\psi^{-1})(z)  (h_y\circ\psi^{-1})(z)
&= \frac {16}{c^2\pi^2} \Im \left( \frac {z^2}
{1+z^4}\right) 
= \frac {16}{c^2\pi^2} \Im \left( \frac {1}
{z^{-2}+z^2}\right).\label{j12.35}
\end{align}
Taking another derivative of \eqref{j12.31} and using \eqref{j12.30},
\begin{align}\notag
((h_{xx}\circ\psi^{-1})(z) - i (h_{yx}\circ\psi^{-1})(z) )(\psi^{-1})'(z)
&= -\frac {4}{c\pi}
\frac {1-z^4}
{(1+z^4)^{3/2}},\\
(h_{xx}\circ\psi^{-1})(z) - i (h_{yx}\circ\psi^{-1})(z) 
&= -\frac {4}{c^2\pi}
\frac {1-z^4}
{1+z^4}, \label{j12.52} \\
 (h_{yx}\circ\psi^{-1})(z) 
&= \frac {4}{c^2\pi} \Im \left(\frac {1-z^4}{1+z^4}\right)
= \frac {4}{c^2\pi} \Im \left(\frac {z^{-2}-z^2}{z^{-2}+z^2}\right).\label{j12.36}
\end{align}

Let $F(z) = h_x(z)/h(z)$ where $z=x+iy$. Then
\begin{align}\label{j12.40}
\frac{\prt F}{\prt y}
= \frac{h h_{xy} -h_xh_y}{h^2}.
\end{align}
We would like to determine the sign of $F_y$ so it will
suffice to analyze the sign of $h h_{xy} -h_xh_y$.

We will change variables as follows,
\begin{align}\label{j13.1}
z= \varphi^{-1}(x+iy), \qquad
z^2 = \rho, \qquad w = \frac{i+\rho}{i-\rho}, \qquad w=t e^{i\alpha},
\end{align}
where $x+iy\in [0,\pi]^2$ and $z,\rho\in D$.

We recall from Remark \ref{rem:1} that for $0\leq r \leq 1$ and integer $k$,
\begin{align*}%\label{j12.8}
 \int_0^{r e^{ik\pi/4}} \frac 1 {\sqrt{1+w^4}} dw
 = e^{ik\pi/4} \int_0^{r } \frac 1 {\sqrt{1\pm t^4}} dt,
\end{align*}
where $+$ corresponds to even $k$ and $-$ to odd $k$. It follows that rays in $D$
emanating from 0 at angles that are multiples of $\pi/4$ are mapped by $\varphi$
onto rays in the square with the same slopes. Hence, the regions between rays
in $D$
are mapped onto the regions between the corresponding rays in the square.
We are concerned with the south-west part of the square. It corresponds
to $z$-values in the south-west part of $D$. As $z\mapsto z^2$ maps this latter region onto the upper half of D, the corresponding $\rho$-values are in the upper half of $D$. 

The function $\rho \mapsto \frac{i+\rho}{i-\rho}$ maps the upper half of $D$  
onto $\{w\in\C\mid \mathrm{Re}(w)>0, |w|>1\}$, so we must have
$t>1$ and $-\pi/2 < \alpha < \pi/2$ in \eqref{j13.1}.
Using (\ref{j13.1}), we can write $z=z(t,\alpha)$, for $z\in D$. Then 

\begin{align}\label{j12.33}
h(z(t,\alpha)) = \frac \alpha \pi +\frac 1 2.
\end{align}

We have
\begin{align*}
&z^2 = \rho = i \left(\frac{w-1}{w+1}\right),\\
&z^2 + z^{-2} = i \left(\frac{w-1}{w+1}\right)- i \left(\frac{w+1}{w-1}\right)
= i \left(
\frac{(w-1)^2 - (w+1)^2}{w^2-1}
\right) = -\frac{4iw}{w^2-1},\\
&z^2 - z^{-2} = i \left(
\frac{(w-1)^2 + (w+1)^2}{w^2-1}
\right)
=2i \left(
\frac{ w^2+1}{w^2-1} \right),\\
& \Im \left(\frac {z^{-2}-z^2}{z^{-2}+z^2}\right)
=  \Im \left(\frac {-2i \left(
\frac{ w^2+1}{w^2-1}\right)}{ -\frac{4iw}{w^2-1}}\right)
=  \Im \left(\frac 1 2 (w + 1/w)\right)
=  \Im \left(\frac 1 2 \left(t e^{i\alpha} + \frac 1 t e^{-i\alpha}\right)\right)\\
&\qquad =\frac 1 2 \left(t - \frac 1 t\right) \sin \alpha,\\
& \Im \left( \frac {1}
{z^{-2}+z^2}\right) 
= \Im \left( \frac {i (w^2-1)}
{4 w}\right)
= \Re \left(\frac 1 4 \left ( w - \frac 1 w \right)   \right)
= \Re \left(\frac 1 4 \left ( t e^{i\alpha} - \frac 1 t e^{-i\alpha} \right)   \right)\\
&\qquad = \frac 1 4 \left(t - \frac 1 t\right) \cos \alpha.
\end{align*}
It follows immediately from these equations and \eqref{j12.33}, \eqref{j12.36} and \eqref{j12.35} that

\begin{align*}
(( h h_{xy} -h_xh_y)\circ\psi^{-1})(z(t,\alpha))
&= \left(\frac \alpha \pi +\frac 1 2\right) \frac{4}{c^2\pi}\cdot
\frac 1 2 \left(t - \frac 1 t\right) \sin \alpha
+ \frac{8}{c^2 \pi^2}\cdot \frac 1 4 \left(t - \frac 1 t\right) \cos \alpha\\
& = \frac{2}{c^2\pi^2}
\left(\left(\alpha + \frac \pi 2\right)  \sin \alpha +  \cos \alpha\right) 
 \left(t - \frac 1 t\right).
\end{align*}
Recall that $t>1$. Hence, $t - \frac 1 t >0$.
If $g(\alpha) = \left(\alpha + \frac \pi 2\right)  \sin \alpha +  \cos \alpha$ then
 $g(-\pi/2) = 0$ and $g'(\alpha ) = \left(\alpha + \frac \pi 2\right)  \cos \alpha>0$
for $-\pi/2 < \alpha< \pi/2$. Therefore, $g(\alpha) \geq 0$
for $-\pi/2\leq \alpha \leq \pi/2$.
We conclude that $(( h h_{xy} -h_xh_y)\circ\psi^{-1})(x+iy)\geq 0$ for
$(x,y)\in[0,\pi/2]^2$
and this implies, in view of \eqref{j12.40}, that $\frac{\prt F}{\prt y}
(x+iy)= \frac{\prt (h_x/h)}{\prt y}(x+iy)  \geq 0$.
\end{proof}

\begin{proposition}\label{j30.2}
For $z=x+i\pi/2$, $0<x<\pi/2$,
\begin{align*}%\label{j14.3}
\frac{h_x(z)}{h(z)} < \cot(\Re z).
\end{align*}
\end{proposition}

\begin{proof}
Let $K(x) = h_x(x+i\pi/2) \sin x - h(x+i\pi/2) \cos x$ for $0\leq x \leq \pi/2$.
It suffices to show that $K$ is negative on this interval, by Proposition \ref{j30.1}.

We will analyze the sign of $K_x(x) = (h_{xx} +h)(x)\sin x$. Since $\sin x>0$
for $0< x \leq \pi/2$, it will suffice to analyze the sign of $h_{xx}+h$.

Note that $h_x(\pi/2+i\pi/2)=0$ by symmetry. By assumption, $h(i\pi/2)=0$. Hence,
\begin{align}\label{j12.50}
K(0)&=   h_x(i\pi/2) \cdot 0 - 0\cdot 1=0,\\
K(\pi/2)& = h_x(\pi/2+i\pi/2) \cdot 1 - \frac{1}{2} \cdot 0 = 0.\label{j12.51}
\end{align}

By \eqref{j12.52} applied to $z=t+i\cdot 0$, 
\begin{align}\label{j12.60}
(h_{xx}\circ\psi^{-1})(z) 
&= -\frac {4}{c^2\pi}
\frac {1-t^4}{1+t^4},
\end{align}
because the right side of (\ref{j12.60}) is real.

Differentiating \eqref{j12.52} and applying \eqref{j12.30} yields
\begin{align*}
((h_{xxx}\circ\psi^{-1})(z) - i (h_{yxx}\circ\psi^{-1})(z) )(\psi^{-1})'(z)
&= \frac {4}{c^2\pi}
\frac {8z^3}
{(1+z^4)^2},  \\
 (h_{xxx}\circ\psi^{-1})(z) - i (h_{yxx}\circ\psi^{-1})(z) 
&= \frac {32}{c^3\pi}
\frac {z^3}
{(1+z^4)^{3/2}},  \\
 (h_{xxx}\circ\psi^{-1})(t) 
&= \frac {32}{c^3\pi}
\frac {t^3}
{(1+t^4)^{3/2}}.
\end{align*}
The last line follows because the right side is real.

By \eqref{j12.31},
\begin{align*}%\label{j12.31}
(h_x\circ\psi^{-1})(t) 
&=  -\frac {4}{c\pi}
\frac t {\sqrt{1+t^4}},
\end{align*}
so
\begin{align*}
(h_{xxx}\circ\psi^{-1})(t)  +
(h_x\circ\psi^{-1})(t) 
&=\frac {32}{c^3\pi}
\frac {t^3}
{(1+t^4)^{3/2}}  -\frac {4}{c\pi}
\frac t {\sqrt{1+t^4}}\\
&= \frac{4}{c\pi} \frac{(-t)}{(1+t^4)^{3/2}} \left(t^4-\frac 8 {c^2} t^2 + 1 \right).
\end{align*}
We want to determine the sign of the last expression so it will
suffice to analyze $p(t):=t^4-( 8/ {c^2}) t^2 + 1 $. The roots of this polynomial must satisfy
\begin{align*}
t^2 = \frac 1 2 \left(\frac 8 {c^2} \pm \sqrt{64/c^4 -4}   \right)
= \frac 4 {c^2} \pm \sqrt{\left(\frac 4 {c^2}\right)^2 -1}.
\end{align*}
Since $c< 2 $  (see Lemma 
 \ref{j14.4}),  these values of $t^2$ are positive and distinct.  Moreover, their product is $1.$ 
 Hence, in $(-1,0)$, the function $p$ has a single root $t_0$. 
%such that $-1< t_0<0$.

Note that
\begin{align*}
(h_{xxx}\circ\psi^{-1})(0)  +
(h_x\circ\psi^{-1})(0) 
&= 0,\\
 (h_{xxx}\circ\psi^{-1})(-1)  +
(h_x\circ\psi^{-1})(-1) 
&= \frac{4}{c\pi} \frac{-1}{(1+1)^{3/2}} \left(\frac 8 {c^2}  - 1 - 1\right)
<0.
\end{align*}
The last two equations along with the fact that $p$ has exactly one root at $t_0$ on $(-1,0)$ imply that This implies that for some $0< x_0 < \pi/2$,
\begin{align*}
&((h_{xx}  +h)_x\circ\psi^{-1})(t) 
\begin{cases}
<0 & \text{  for  } -1<t<t_0,\\
>0 & \text{  for  } t_0<t<0,
\end{cases}\\
&(h_{xx}  +h)_x(x) 
\begin{cases}
<0 & \text{  for  } 0<x<x_0,\\
>0 & \text{  for  } x_0<x<\pi/2.
\end{cases}
\end{align*}
By \eqref{eq: BVP_h} and \eqref{j12.60},
\begin{align*}%\label{j12.60}
((h_{xx}+h)\circ\psi^{-1})(-1) 
&= -\frac {4}{c^2\pi}
\frac {1-1}{1+1}+0=0,\\
((h_{xx}+h)\circ\psi^{-1})(0) 
&= -\frac {4}{c^2\pi}+ \frac 1 2 >0.
\end{align*}
These observations imply that the function $(h_{xx}+h)(x)$ 
decreases from 0 to a minimum at $x_0$ and then increases to a positive
maximum at $\pi/2$. Thus for some $0< x_1< \pi/2$,
$K_x(x)=(h_{xx}+h)(x)\sin x$ is negative for $0<x< x_1$ and positive
for $x_1< x< \pi/2$. In view of \eqref{j12.50}-\eqref{j12.51}, $K$
is negative on the interval $[0,\pi/2]$.
\end{proof}

\begin{proof}[Proof of Theorem \ref{j19.7}]
The theorem follows from  Propositions \ref{j30.1} and \ref{j30.2}.
\end{proof}
\
\section{Estimates based on the stationary distribution}\label{sec:stat}

We will show that for sufficiently small $\eps>0$, for some  $c_{\eps}<C_{\eps}$, the long-term occupation measure of $(0, \eps)$ for Brownian motion conditioned to stay inside $(0,\pi)$ is smaller than $c_\eps$ while the long-term occupation measure of $(0, \eps)$  for the spine is larger than $C_\eps$. 

We start with an elementary result based on well known estimates of harmonic measure and complex analytic methods.

Recall that the harmonic measure is the exit distribution of a Brownian motion, with a given starting (base) point $z$.
More precisely, the harmonic measure of a subset of the boundary of a bounded domain $D$ in  $\R^{2}$ is the probability that a Brownian motion started  at $z\in D$ hits that subset of the boundary.
For a domain $D$, $A\subset \prt D$ and $z\in D$, let $\mu_z(A)$ denote the harmonic measure of $A$ in $D$ with the base point $z$. 

\begin{lemma}\label{j4.1}
(i) If $D= \{v\in\C: |\Im v - a| < r, \Re v < b\}$, $A=\{v\in \prt D: \Re v = b\}$ and $z=c+ ai$ where $c< b$, then
\begin{align*}
\mu_z(A)= \frac 4 \pi \arctan\left( \exp\left(-\frac{\pi (b-c)}{2r}\right)\right).
\end{align*}

(ii) If $D = \{v\in\C: \Re v >0, \Im v > 0, |v| < \pi\}$, $A= \{v\in \prt D: |v| = \pi\}$ and $z=x+yi\in D$, then
\begin{align*}
\mu_z(A) \leq \frac 4 \pi \arctan((x^2+y^2)/\pi^2).
\end{align*}

\end{lemma}

\begin{proof} (i)
Our formula is obtained by scaling from the formula  stated on the line after (5.3) on page 144 in 
\cite{GM}.

(ii)  If $x=y$, we can transform the problem to that in part (i) by using the conformal mapping $z\to \log z$ and then applying conformal invariance of harmonic measure (use $a=\pi/4, b=\log\pi, c=\log(\sqrt{2}x),r=\pi/4$). In the case $x=y$, the formula holds with the equality sign. To extend the formula to the case $x\ne y$, we note that the function $z\to \mu_z(A)$ takes the highest value in the middle of the arc $\{z\in D: |z|^2 =x^2+y^2\}$. The last remark follows from (5.3) on page 144 in  \cite{GM}.
\end{proof}

\begin{proof}[Proof of Theorem \ref{s26.1}] Since the normalized principal Dirichlet eigenfunction of the Laplacian in $(0,\pi)$ is $\phi(x)=\sqrt{2/\pi}\sin(x)$,
the stationary density for Brownian motion $W_t$ conditioned to stay in $(0,\pi)$ forever is $\phi^2(x)$=$(2/\pi) \sin ^2 x $. Hence, by the ergodic theorem, a.s.,
\begin{align}\label{j3.1}
\lim_{t\to\infty} \frac 1 t \int_0^t \bone_{(0,\eps)}(W_s) ds
= \int_0^\eps (2/\pi) \sin ^2 x dx = 
\frac 1 \pi (\eps - (1/2) \sin(2\eps))
= \frac{2 \eps ^3}{3 \pi }-\frac{2 \eps ^5}{15 \pi }+o\left(\eps ^6\right).
\end{align}

Let $G(x,y)$ denote the Green function for Brownian motion in $(0,\pi)$ killed upon exiting the interval.
According to \cite[Thm. 1.5]{MK}, $G(x,y)$, appropriately normalized, is the stationary density for the two-particle Fleming-Viot process in $(0,\pi)$ driven by Brownian motion. 

We have
\begin{align*}
G(x,y) =
\begin{cases}
\frac 2 {\pi^2} y (\pi-x) & \text{ if } 0<y < x< \pi,\\
\frac 2 {\pi^2} x (\pi - y) & \text{ if } 0<x < y< \pi.
\end{cases}
\end{align*}
The normalizing constant $\frac 2 {\pi^2}$  is chosen so that $\int_0^\pi \int_0^\pi G(x,y) dx dy=1$.

Suppose that the process is in the stationary regime.
Consider $x,y\in (0, \pi/2)$. At any time $t\geq 0$, the probability that the spine is in $dx$ and it is $X_t$, and $Y_t\in dy$, is equal to $G(x,y)dxdy$ times the probability $p(x,y)$ that $Y$ will exit $(0,\pi)$ before $X$ does.

Suppose that $y \leq x^{1/2}$.
The probability $p(x,y)$ is bounded below by the probability $p_1$ that
two-dimensional Brownian motion starting from $(x,y)$ will hit the $x$-axis before it hits the $y$-axis minus the probability $p_2$ that $(X_t,Y_t)$ will hit the circle centered at 0 with radius $\pi$ before $X$ or $Y$ exit the first quadrant.
We have $p_1=(2/\pi)\arctan(x/y)$ because this is the unique bounded harmonic function in the first quadrant with boundary values 0 on the vertical axis and 1 on the horizontal axis.
By Lemma \ref{j4.1} (ii),
\begin{align*}
p_2\leq\frac 4 \pi \arctan((x^2+y^2)/\pi^2).
\end{align*}
Hence, in this case,
\begin{align}\label{j4.2}
p(x,y) &\geq p_1-p_2 \geq 
\frac2\pi\arctan(x/y) - \frac 4 \pi \arctan((x^2+y^2)/\pi^2).
\end{align}
For sufficiently small $x>0$, since we are assuming that  $y \leq x^{1/2}$,
\begin{align*}
\frac 1 2 \cdot
\frac2\pi\arctan(x/y) \geq \frac 1 2 \cdot \frac2\pi\arctan(x^{1/2})
\geq \frac 4 \pi \arctan((x^2+x)/\pi^2) \geq
 \frac 4 \pi \arctan((x^2+y^2)/\pi^2).   
\end{align*}
This and \eqref{j4.2} imply that for sufficiently small $x>0$,
\begin{align*}%\label{j4.3}
p(x,y) &\geq  
\frac1\pi\arctan(x/y) .
\end{align*}

Factor 2 in the following formula is to account for both $X$ and $Y$ possibly being the spine.
In the calculation below we will use the  estimates $\arctan(x/y) \geq (\pi/4) x/y$ for $y\geq x$,  and $\pi-y > \pi-1$, for $y\in (0,1)$. 
The two-particle process $\calV(t) $ in \eqref{j18.1} is ergodic by  \cite[Thm. 1.5]{MK}.
It follows from the ergodic theorem 
that for sufficiently small $\eps\in(0,1)$, a.s.,
\begin{align*}
\lim_{t\to\infty}& \frac 1 t \int_0^t \bone_{(0,\eps)}(J_s) ds
=\int_0^\eps
2\int_0^1 G(x,y) p(x,y) d y dx\\
&\geq \int_0^\eps
2\int_x^{x^{1/2}} \frac 2 {\pi^2} x (\pi-y) 
\frac 1 \pi \arctan(x/y)dy dx\\
&\geq c_1 \int_0^\eps
\int_x^{x^{1/2}}  \frac{x^2}{y}dy dx
=\frac{c_1}{6} \eps^3\left(\log\frac{1}{\eps}+\frac{1}{3}\right).
 \end{align*} 
This is larger than the quantity in \eqref{j3.1} for small $\eps>0$. 
It follows that the distribution of the spine and the distribution of Brownian motion conditioned to stay in $(0,\pi)$ forever are mutually singular.
\end{proof}

%\bibliographystyle{amsplain}
%\bibliography{twospine}

\providecommand{\bysame}{\leavevmode\hbox to3em{\hrulefill}\thinspace}
\providecommand{\MR}{\relax\ifhmode\unskip\space\fi MR }
% \MRhref is called by the amsart/book/proc definition of \MR.
\providecommand{\MRhref}[2]{%
  \href{http://www.ams.org/mathscinet-getitem?mr=#1}{#2}
}
\providecommand{\href}[2]{#2}

\end{document}